\documentclass[12pt]{amsart}
\usepackage{eurosym}
\usepackage{amssymb}
\usepackage{amsmath}
\usepackage{amsfonts}
\usepackage{graphicx,color}
\usepackage{xcolor}

\newtheorem{theorem}{Theorem}
\theoremstyle{plain}

\newtheorem{proposition}{Proposition}
\newtheorem{remark}{Remark}

\numberwithin{equation}{section}

\begin{document}
\title[Singular elliptic problems]{Nodal and constant sign solutions for
singular elliptic problems}

\begin{abstract}
We establish the existence of multiple solutions for singular quasilinear
elliptic problems with a precise sign information: two opposite constant
sign solutions and a nodal solution. The approach combines
sub-supersolutions method and Leray-Schauder topological degree involving
perturbation argument.
\end{abstract}

\author[D. Motreanu]{Dumitru Motreanu}
\address{Universit\'e de Perpignan, D\'epartement de Math\'ematiques, 66860
Perpignan, France}
\email{motreanu@univ-perp.fr}
\author[A. Moussaoui]{Abdelkrim Moussaoui}
\address{Applied Mathematics Laboratory, Faculty of Exact Sciences \\
and Biology Department, Faculty of Natural and Life Sciences\\
A. Mira Bejaia University, Targa Ouzemour, 06000 Bejaia, Algeria}
\email{abdelkrim.moussaoui@univ-bejaia.dz}
\maketitle

\section{Introduction}

Let $\Omega \subset \mathbb{R}^{N}$ $\left( N\geq 2\right) $ be a bounded
domain with $C^{1}$ boundary $\partial \Omega $. We consider the singular
elliptic problem 
\begin{equation*}
(\mathrm{P})\qquad \left\{ 
\begin{array}{ll}
-\Delta _{p}u=f(x,u) & \text{in}\;\Omega , \\ 
u=0 & \text{on}\;\partial \Omega ,%
\end{array}%
\right.
\end{equation*}%
where $\Delta _{p}$ stands for the $p$-Laplacian operator on $%
W_{0}^{1,p}(\Omega )$ with $1<p<\infty $ and the reaction term $f(x,u$) is
described by a Carath\'{e}odory function $f:\Omega \times (\mathbb{R}%
\backslash \{0\})\rightarrow \mathbb{R}$ that can exhibit singularity when $%
u $ approaches zero. This occurs under the following hypothesis on the
nonlinearity $f(x,s)$ reflecting the singular character of problem $(\mathrm{%
P})$:

\begin{equation*}
\lim_{s\rightarrow 0}f(x,s)=+\infty \text{ \ uniformly for a.e. }x\in \Omega
.
\end{equation*}%
For more information on the study of singular problems we refer to \cite{DM,
DM2, H, KM, LM, MMP, MMM, MM1, MM2, MM3, M} and the references therein.

\bigskip

The presence of a singularity in $(\mathrm{P})$ represents a major obstacle
to overcome, especially since nonpositive solutions are addressed. Namely,
beside positive and negative solutions, we also discuss the existence of
nodal solutions for $(\mathrm{P})$ that are sign changing functions and
inevitably go through the singularity. As far as we know, such solutions for
problems involving singularities have been rarely investigated in the
literature. Actually, \cite{M} is the only paper that has considered this
issue. The existence of nodal solutions is established for a class of
semilinear singular system by means of the trapping region formed by pairs
of appropriately defined sign changing sub-supersolutions. Exploiting
spectral properties of Laplacian operator as well as adequate truncation,
the author establishes that the nodal solution vanishes on negligible sets.
This is an essential point enabling nodal solutions investigation for
singular problems.

In line with \cite{M}, by a solution of problem $(\mathrm{P})$ we mean any $%
u\in W_{0}^{1,p}\left( \Omega \right) $ such that $f(x,u)\varphi \in
L^{1}(\Omega )$ and 
\begin{equation*}
\int_{\Omega }\left\vert \nabla u\right\vert ^{p-2}\nabla u\nabla \varphi \
dx=\int_{\Omega }f(x,u)\varphi \ dx,
\end{equation*}%
for all $\varphi \in W_{0}^{1,p}(\Omega )$.

Our main goal is to establish multiplicity result for singular quasilinear
elliptic problem $(\mathrm{P})$, with a precise sign information. We provide
three solutions for $(\mathrm{P})$, two of which are of opposite
constant-sign solutions while the third one is nodal. Our approach combines
sub-supersolution method and topological degree theory. We first establish
the existence of opposite constant-sign solutions $u_{+},u_{-}\in \mathcal{C}%
^{1,\tau }(\overline{\Omega })$. They are located in positive and negative
intervals formed by two opposite constant sign sub-supersolutions pairs,
constructed through a choice of suitable functions with an adjustment of
adequate constants. Then, focusing on the aforementioned intervals, we show
the existence of a minimal positive solution $u_{+}^{\ast }$ and a maximal
negative solution $u_{-}^{\ast }$ for problem $(\mathrm{P})$. The argument
is essentially based on the Zorn's Lemma as well as the $S_{+}$-property of
the negative $p$-Laplacian operator on $W_{0}^{1,p}(\Omega )$.

A significant feature in the present work concerns the finding of a nodal
solution for singular problem $(\mathrm{P})$. It is a third solution for $(%
\mathrm{P})$ achieved via the Leray-Schauder degree theory. The latter is
applied to a perturbed system $(\mathrm{P}_{\varepsilon }),$ depending on a
parameter $\varepsilon >0,$ whose study is relevant for problem $(\mathrm{P})
$. Precisely, we prove that the degree on a ball $\mathcal{B}%
_{R_{\varepsilon }}(0)$, centered at $0$ of radius $R_{\varepsilon }$,
encompassing all potential solutions of $(\mathrm{P}_{\varepsilon })$ is $0,$
while the degree on $\mathcal{B}_{R_{\varepsilon }}(0)\backslash \mathcal{A}%
_{R_{\varepsilon }},$ where $\mathcal{A}_{R_{\varepsilon }}$ is a set
containing only solutions lying inside the interval $]u_{-}^{\ast
},u_{+}^{\ast }[,$ is equal to $1$. By the excision property of
Leray-Schauder degree, this leads to the existence of a nontrivial solution $%
u_{\varepsilon }^{\ast }$ for $(\mathrm{P}_{\varepsilon })$\ in $\mathcal{A}%
_{R_{\varepsilon }}.$ Then, a solution $u^{\ast }\in ]u_{-}^{\ast
},u_{+}^{\ast }[$ of $(\mathrm{P})$ is derived by passing to the limit as $%
\varepsilon \rightarrow 0.$ The argument is based on a priori estimates,
dominated convergence theorem as well as $S_{+}$-property of the negative $p$%
-Laplacian. Due to the extremality of constant-sign solutions $u_{-}^{\ast }$
and $u_{+}^{\ast },$ then $u^{\ast }$ is a third solution of $(\mathrm{P})$
that changes sign.

The rest of this article is organized as follows. Section \ref{S2} deals
with opposite constant sign solutions while Section \ref{S3} provides a
nodal solution for problem $(\mathrm{P})$.

\section{Constant sign solutions}

\label{S2}

Given $1<p<\infty$, the spaces $L^{p}(\Omega )$ and $W_{0}^{1,p}(\Omega)$
are endowed with the usual norms $\Vert u\Vert_{p}=(\int_{\Omega }|u|^{p}\
dx)^{1/p}$ and $\Vert u\Vert _{1,p}=(\int_{\Omega }|\nabla u|^{p}\ dx)^{1/p}$%
, respectively. For a later use, set $p^{\prime }=\frac{p}{p-1}$. We will
also use the spaces $C(\overline{\Omega })$ and $C_{0}^{1,\beta }(\overline{%
\Omega })=\{u\in C^{1,\beta }(\overline{\Omega } ):u=0\ 
\mbox{on
$\partial\Omega$}\}$ with $\beta \in (0,1)$.

Let us denote by $\lambda _{1,p}$ the first eigenvalue of $-\Delta _{p}$ on $%
W_{0}^{1,p}(\Omega )$ which is given by 
\begin{equation}
\lambda _{1,p}=\inf_{u\in W_{0}^{1,p}(\Omega )\setminus \{0\}}\frac{%
\int_{\Omega }|\nabla u|^{p}\,\mathrm{d}x}{\int_{\Omega }|u|^{p}\,\mathrm{d}x%
}.  \label{71}
\end{equation}
Let $\phi_{1,p}$ be the positive normalized eigenfunction of $-\Delta _{p}$
corresponding to $\lambda _{1,p}$, that is, 
\begin{equation*}
-\Delta _{p}\phi _{1,p}=\lambda _{1,p}\phi _{1,p}^{p-1}\ \text{in }\Omega ,\
\ \phi _{1,p}=0\ \text{on }\partial \Omega,
\end{equation*}
$\phi _{1,p}>0$ in $\Omega $, and $\Vert \phi _{1,p}\Vert _{p}=1$. Recall
that there exists a constant $c_{0}>0$ such that 
\begin{equation}
\phi _{1,p}(x)\geq c_{0}d(x)\text{ for all }x\in \Omega ,  \label{14}
\end{equation}
where $d(x)$ denotes the distance from $x\in \overline{\Omega }$ to the
boundary $\partial \Omega$.

For $u,v\in C^{1}(\overline{\Omega })$, the notation $u\ll v$ means 
\begin{equation*}
\begin{array}{c}
u(x)<v(x),\,\,\forall x\in \Omega,\,\,\mbox{and}\,\,\,\frac{\partial v}{%
\partial \eta }<\frac{\partial u}{\partial \eta }\,\,\,\mbox{on}%
\,\,\,\partial \Omega,%
\end{array}%
\end{equation*}
where $\eta $ is the outward unit normal to $\partial \Omega$.

\begin{remark}
\label{R1} Under assumption $(\mathrm{H.1})$, any solution $u\in
W_{0}^{1,p}(\Omega )$ of $(\mathrm{P})$ satisfies $u(x)\neq 0$ for a.e. $%
x\in \Omega $. Otherwise, if $u$ vanishes on a set of positive measure, then 
$(\mathrm{P})$ and hypothesis $(\mathrm{H}.1)$ are contradictory.
\end{remark}

\subsection{Opposite constant sign solutions}

For a fixed $\delta >0$, set $\Omega _{\delta }:=\left\{ x\in \Omega
:d(x)<\delta \right\} $. Let $y,y_{\delta }\in \mathcal{C}^{1}\left( 
\overline{\Omega }\right) $ be the unique solutions of the Dirichlet
problems 
\begin{equation}
-\Delta _{p}y=d(x)^{\alpha }\text{ in }\Omega ,\text{ \ }y=0\text{ on }%
\partial \Omega ,  \label{20}
\end{equation}%
\begin{equation}
-\Delta _{p}y_{\delta }=\left\{ 
\begin{array}{ll}
d(x)^{\alpha } & \text{in \ }\Omega \backslash \overline{\Omega }_{\delta },
\\ 
-1 & \text{in \ }\Omega _{\delta },%
\end{array}%
\right. ,\text{ \ }y_{\delta }=0\text{ on }\partial \Omega ,  \label{20*}
\end{equation}
where $\alpha \in (-1,0)$. They satisfy 
\begin{equation}
c^{-1}d(x)\leq y_{\delta }(x)\leq y(x)\leq cd(x)\text{ in }\Omega,
\label{21}
\end{equation}
with a constant $c>1$ (see \cite[Lemma 3.1]{DM}).

For a constant $C>1$, we also pose 
\begin{equation}  \label{24}
\overline{u}=Cy\text{ \ and \ }\underline{u}=C^{-1}y_{\delta}.
\end{equation}
Due to \eqref{21}, this gives 
\begin{equation}  \label{18}
\overline{u}\geq \underline{u}\ \text{ in }\ \overline{\Omega },
\end{equation}
so we can consider the ordered interval 
\begin{equation*}
\left[ \underline{u},\overline{u} \right]=\{v\in W_{0}^{1,p}(\Omega):%
\underline{u}(x)\leq v(x)\leq \overline{u}(x)\ \text{a.e.}\ x\in\Omega \}.
\end{equation*}

We further assume.

\begin{description}
\item[$(\mathrm{H}.2)$] \textit{There exist constants }$M>0$\textit{\ and }$%
-1<\alpha <0$\textit{\ such that} 
\begin{equation*}
f(x,s)\leq M(1+s^{\alpha })\text{ \ and \ }f(x,-s)\geq -M(1+|s|^{\alpha }),
\end{equation*}
\ for a.e\textit{. }$x\in \Omega ,$ all $s>0.$

\item[$(\mathrm{H}.3)$] \textit{There exist constants }$\beta <0<m$\textit{\
such that} 
\begin{equation*}
\alpha \geq \beta >-\min \{1,p-1\},
\end{equation*}
\begin{equation*}
f(x,s)\geq ms^{\beta }\text{ \ and \ }f(x,-s)\leq -m|s|^{\beta },
\end{equation*}
for a.e\textit{. }$x\in \Omega ,$ all $s>0.$
\end{description}

\begin{theorem}
\label{T1} Assume that $(\mathrm{H.2})$ and $(\mathrm{H.3})$ hold. Then
problem $(\mathrm{P})$ admits a positive solution $u_{+}\in \mathcal{C}%
^{1,\tau }(\overline{\Omega })$ and a negative solution $u_{-}\in \mathcal{C}%
^{1,\tau }(\overline{\Omega })$ within the ordered intervals $\left[ 
\underline{u},\overline{u}\right] $ and $\left[ -\overline{u},-\underline{u}%
\right] ,$ respectively. Moreover, every positive solution $u_{+}$ and
negative solution $u_{-}$ of $(\mathrm{P})$ within $\left[ 0,\overline{u}%
\right] \cap \mathcal{C}^{1}(\overline{\Omega })$ and $\left[ -\overline{u},0%
\right] \cap \mathcal{C}^{1}(\overline{\Omega })$, respectively, satisfy 
\begin{equation}
\underline{u}(x)\leq u_{+}(x)\text{ \ and \ }u_{-}(x)\leq -\underline{u}(x),%
\text{ }\forall x\in \Omega .  \label{p1}
\end{equation}
\end{theorem}

\begin{proof}
By (\ref{21}) one has 
\begin{eqnarray*}
&&C^{p-1}d(x)^{\alpha}\geq M\left( \left\Vert d\right\Vert_{\infty
}^{-\alpha }+(Cc^{-1})^{\alpha }\right) d(x)^{\alpha } \\
&&\geq M\left( d(x)^{-\alpha }+(Cc^{-1})^{\alpha }\right) d(x)^{\alpha } \\
&&\geq M(1+(Cc^{-1}d(x))^{\alpha })\geq M(1+(Cy(x))^{\alpha })\text{ \ for
all } x\in \overline{\Omega },
\end{eqnarray*}
provided $C>0$ is sufficiently large. Thus, from $(\mathrm{H.2})$, (\ref{20}%
) and (\ref{24}), it follows that 
\begin{eqnarray}  \label{5}
-\Delta_{p}\overline{u}\geq M(1+\overline{u}^{\alpha })\geq f(x,\overline{u}
)\text{ in }\Omega
\end{eqnarray}
and 
\begin{equation}  \label{6}
-\Delta _{p}(-\overline{u})\leq -M(1+\overline{u}^{\alpha })\leq f(x,-%
\overline{u})\text{ in }\Omega.
\end{equation}
From \eqref{20*}, \eqref{21}, \eqref{24}, and $(\mathrm{H.3})$, we get 
\begin{eqnarray}  \label{2}
&&-\Delta _{p}\underline{u}\leq C^{-(p-1)}d(x)^{\alpha }\leq C^{-\beta
}m(cd(x))^{\beta } \\
&&\leq C^{-\beta }my(x)^{\beta }=m\underline{u}(x)^{\beta}\leq f(x,%
\underline{u}(x)) \text{ in }\Omega  \notag
\end{eqnarray}
and 
\begin{equation}  \label{19*}
-\Delta_{p}(-\underline{u}) \geq -m\underline{u}(x)^{\beta}\geq f(x,-%
\underline{u}(x))\text{ in }\Omega
\end{equation}
provided $C>1$ is sufficiently large.

In view of \eqref{18}, \eqref{5}-\eqref{19*}, we deduce that $(\underline{u}%
, \overline{u})$ and $(-\overline{u},-\underline{u})$ form
sub-supersolutions pairs for $(\mathrm{P})$. Then, the sub-supersolution
method (see \cite[Theorem 2]{KM}) leads to the existence of a positive
solution $u_{+}\in \left[ \underline{u},\overline{u}\right] \cap \mathcal{C}%
^{1,\tau }(\overline{\Omega })$ and a negative solution $u_{-}\in \left[ -%
\overline{u},-\underline{u}\right] \cap \mathcal{C}^{1,\tau }(\overline{%
\Omega })$, with some $\tau \in (0,1)$.

Let $u_{+}\in \left[ 0,\overline{u}\right] \cap \mathcal{C}^{1}(\overline{%
\Omega })$ be a positive solution of $(\mathrm{P})$. By $(\mathrm{H.3})$, %
\eqref{21} and \eqref{24}, we get 
\begin{eqnarray*}
-\Delta _{p}u_{+} &=&f(x,u_{+})\geq mu_{+}^{\beta }\geq m\overline{u}^{\beta
}\geq m(Ccd(x))^{\beta } \\
&\geq &C^{-(p-1)}d(x)^{\alpha }\geq -\Delta _{p}\underline{u}(x)\text{ \ in }%
\Omega ,
\end{eqnarray*}
provided $C>1$ is sufficiently large and $\delta >0$ sufficiently small in %
\eqref{20*}. By the weak comparison principle, we infer that property (\ref%
{p1}) holds true. Using a similar argument, we show the corresponding
property for a negative solution.
\end{proof}

\subsection{Extremal solutions}

\begin{theorem}
\label{T2} Under assumptions $(\mathrm{H.2})$ and $(\mathrm{H.3})$, problem $%
(\mathrm{P})$ admits a smallest positive solution $u_{+}^{\ast }\in \mathcal{%
C}^{1}(\overline{\Omega })$ within $[\underline{u},\overline{u}]$ and a
biggest negative solution $u_{-}^{\ast }\in \mathcal{C}^{1}(\overline{\Omega 
})$ within $[-\overline{u},-\underline{u}]$. In addition, it holds 
\begin{equation}
\begin{array}{c}
\underline{u}\ll u_{+}^{\ast }\text{ \ and }-\underline{u}\gg u_{-}^{\ast
}\quad \text{in}\ \Omega .%
\end{array}
\label{11*}
\end{equation}
\end{theorem}

\begin{proof}
We only prove the existence of a smallest positive solution $u_{+}^{\ast
}\in \mathcal{C}^{1}(\overline{\Omega })$ within $[\underline{u},\overline{u}%
]$. The proof for the biggest negative solution in $[-\overline{u},-%
\underline{u}]$ can be carried out in a similar way.

Denote by $\mathcal{S}$ the set of all functions $w\in \lbrack \underline{u},%
\overline{u}]$ that are solutions of $(\mathrm{P})$. From Theorem \ref{T1}
it is known that $\mathcal{S}$ is not empty. We claim that $\mathcal{S}$ is
downward directed. To this end, let $u_{1},u_{2}\in \mathcal{S}$. Since $(%
\underline{u},u_{1})$ and $(\underline{u},u_{2})$ form pairs of
sub-supersolutions for $(\mathrm{P})$, if we set $\tilde{u}=\min
\{u_{1},u_{2}\}$, by virtue of \cite[Theorem 3.20]{CLM}, $(\underline{u},%
\tilde{u})$ is a pair of sub-supersolution, too. Then, owing to \cite[%
Theorem 2]{KM}, there exists a solution of $(\mathrm{P})$ in $[\underline{u},%
\tilde{u}]\cap \mathcal{C}^{1}(\overline{\Omega })$, which proves the claim.

Now we show that Zorn's lemma can be applied in $\mathcal{S}$. In this
respect, let us consider a chain $\mathcal{C}$ in $\mathcal{S}$ with respect
to the order $\geq $. There is a sequence $\{u_{k}\}_{k\geq 1}\subset 
\mathcal{C}$ such that $\inf C=\inf_{k\geq 1}u_{k}$ (see \cite[page 336]%
{DunSch}) and we can assume that the sequence $\{u_{k}\}_{k\geq 1}$ is
decreasing. If $\hat{u}:=\inf \mathcal{C}$, one has $u_{k}\rightarrow \hat{u}
$ a.e. in $\Omega $ and $\hat{u}\in \lbrack \underline{u},\overline{u}]$.
Having that each $u_{k}$ is a solution of $(\mathrm{P}),$ from $(\mathrm{H.2}%
)$, (\ref{21}) and (\ref{24}), we obtain 
\begin{equation*}
\left\Vert \nabla u_{k}\right\Vert _{p}^{p}=\int_{\Omega }f(x,u_{k})u_{k}\
dx\leq M\int_{\Omega }(u_{k}+u_{k}^{\alpha +1})\text{ }dx\leq
C_{0}(\left\Vert \overline{u}\right\Vert _{\infty }+\left\Vert \overline{u}%
\right\Vert _{\infty }^{1+\alpha }),
\end{equation*}
with a constant $C_{0}>0$. Therefore $\{u_{k}\}$ is bounded in $%
W_{0}^{1,p}(\Omega )$ and so up to a subsequence we have $%
u_{k}\rightharpoonup \hat{u}\text{ in }W_{0}^{1,p}(\Omega )$. Using that 
\begin{equation*}
\langle -\Delta _{p}u_{k},u_{k}-\hat{u}\rangle =\int_{\Omega
}f(x,u_{k})(u_{k}-\hat{u})\ dx,
\end{equation*}%
$(\mathrm{H.2})$, (\ref{21}), (\ref{24}), $u_{k}\in \mathcal{S}$, we find 
\begin{eqnarray*}
|f(x,u_{k})(u_{k}-\hat{u})| &\leq &M(1+u_{k}^{\alpha })|u_{k}-\hat{u}|\leq
M(1+(Cc)^{-\alpha }d(x)^{\alpha })|u_{k}-\hat{u}| \\
&\leq &2M(1+(Cc)^{-\alpha }d(x)^{\alpha })\left\Vert \overline{u}\right\Vert
_{\infty }\text{ \ a.e. in }\Omega.
\end{eqnarray*}
In view of \cite[Lemma, page 726]{LM}, we have $d(x)^{\alpha }\in
L^{1}(\Omega )$. Then Lebesgue's dominated convergence theorem implies that 
\begin{equation*}
\underset{k\rightarrow \infty }{\lim }\langle -\Delta _{p}u_{k},u_{k}-\hat{u}%
\rangle =0,
\end{equation*}
which allows us to invoke the $S_{+}$-property of $-\Delta _{p}$ on $%
W_{0}^{1,p}(\Omega )$ that gives $u_{k}\longrightarrow \hat{u}$ in $%
W_{0}^{1,p}(\Omega )$. We conclude that $\hat{u}=\inf \mathcal{C}\in\mathcal{%
S}$. Zorn's Lemma can be applied which provides a minimal element $%
u_{+}^{\ast }$ of $\mathcal{S}$. Since $\mathcal{S}\subset \lbrack 
\underline{u},\overline{u}]$, by $(\mathrm{H.2})$, (\ref{21}) and (\ref{24}%
), we infer from the regularity result in \cite[Lemma 3.1]{H} that $%
u_{+}^{\ast }\in \mathcal{C}^{1,\tau }(\overline{\Omega })$ for some $\tau
\in(0,1)$.

We claim that $u_{+}^{\ast }$ is the smallest solution of $(\mathrm{P})$ in $%
\mathcal{S}$. Indeed, if $u\in \mathcal{S}$, there is $\tilde{u}\in \mathcal{%
S}$ with $\tilde{u}\leq u_{+}^{\ast }$ and $\tilde{u}\leq u$ because $%
\mathcal{S}$ is downward directed. Recalling that $u_{+}^{\ast }$ is a
minimal element of $\mathcal{S}$, this results in $u_{+}^{\ast }=\tilde{u}%
\leq u$. This proves the claim.

What remains to show is the first inequality in (\ref{11*}). The second one
follows similarly. By $(\mathrm{H.3})$, \eqref{21} and \eqref{24}, for each
compact set $\mathrm{K}\subset \Omega ,$ there is a constant $\sigma =\sigma
(\mathrm{K})>0$ such that 
\begin{eqnarray*}
-\Delta _{p}u_{+}^{\ast } &=&f(x,u_{+}^{\ast })\geq m(u_{+}^{\ast })^{\beta
}\geq m\overline{u}^{\beta }\geq m(Ccd(x))^{\beta } \\
&>&C^{-(p-1)}d(x)^{\alpha }+\sigma \geq -\Delta _{p}\underline{u}(x)+\sigma 
\text{ \ in }\mathrm{K}
\end{eqnarray*}
provided $C>1$ is sufficiently large and $\delta >0$ sufficiently small in %
\eqref{20*}. By the strong comparison principle \cite[Proposition $2.6$]{AR}
, we infer that property (\ref{11*}) holds true. A quite similar argument
shows the corresponding property for a biggest negative solution $%
u_{-}^{\ast }$. This completes the proof.
\end{proof}

\section{Nodal solutions}

\label{S3}

Our main result is stated as follows.

\begin{theorem}
\label{T3} Under assumptions $(\mathrm{H.1})$, $(\mathrm{H.2})$ and $(%
\mathrm{H.3})$, problem $( \mathrm{P})$ possesses a solution $u_{\ast}\in
W_{0}^{1,p}\left( \Omega \right) $ such that $u_{\ast }\neq u_{+}$ and $%
u_{\ast}\neq u_{-}$. Moreover, $u_{\ast }$ is a nodal (sign-changing)
solution.
\end{theorem}

We will make use of the Leray-Schauder degree theory. For any $\varepsilon
\in (0,1)$, we introduce the function 
\begin{equation}
\gamma _{\varepsilon }(s):=\varepsilon \text{ }sgn(s),\quad \forall s\in 
\mathbb{R},  \label{12}
\end{equation}
and corresponding to it the regularized problem 
\begin{equation*}
(\mathrm{P}_{\varepsilon })\qquad \left\{ 
\begin{array}{ll}
-\Delta _{p}u=f(x,u+\gamma _{\varepsilon }(u)) & \text{in}\;\Omega , \\ 
u=0 & \text{on}\;\partial \Omega.%
\end{array}
\right.
\end{equation*}
Our goal is to prove that $(\mathrm{P}_{\varepsilon })$ admits a solution
within $[-\underline{u},\underline{u}]$ and then, passing to the limit as $%
\varepsilon \rightarrow 0$, to get the existence of the desired solution $%
u_{\ast }$ for problem $(\mathrm{P})$.

Theorem \ref{T1} provides two opposite constant-sign solutions $%
u_{\varepsilon ,+}$ and $u_{\varepsilon ,-}$ for problem $(\mathrm{P}%
_{\varepsilon})$ in $\mathcal{C}^{1,\tau }(\overline{\Omega})$ with some $%
\tau \in (0,1)$ satisfying (\ref{p1}). In particular, the set of solutions $%
u_{\varepsilon }\in \mathcal{C}^{1,\tau}(\overline{\Omega })$ to $(\mathrm{P}%
_{\varepsilon })$ is nonempty.

For any $R>0$, set 
\begin{equation*}
\mathcal{A}_{R}=\left\{ u\in \mathcal{B}_{R}(0)\,:-\underline{u}\leq u\leq 
\underline{u}\right\}
\end{equation*}%
and 
\begin{equation*}
\mathcal{B}_{R}(0)=\left\{ u\in \mathcal{C}^{1}(\overline{\Omega}
):\,\left\Vert u\right\Vert _{\mathcal{C}^{1}(\overline{\Omega })}<R\right\}
.
\end{equation*}

We are going to establish the solvability of problem $(\mathrm{P}%
_{\varepsilon})$.

\begin{theorem}
\label{T4} Under assumptions $(\mathrm{H.1})$, $(\mathrm{H.2})$ and $(%
\mathrm{H.3})$, the regularized problem $(\mathrm{P}_{\varepsilon })$
possesses a solution $u_{\varepsilon }\in ]u_{-}^{\ast },u_{+}^{\ast }[$ for
all $\varepsilon \in (0,1)$, where $u_{-}^{\ast }$ and $u_{+}^{\ast }$ are
the extremal constant-sign solutions found in Theorem \ref{T2}.
\end{theorem}

\subsubsection{\textbf{The topological degree on a ball }$\mathcal{B}_{R_{%
\protect\varepsilon }}(0)$}

We introduce the following homotopy class of problems 
\begin{equation*}
(\mathrm{P}_{\varepsilon ,t})\qquad -\Delta _{p}u=\mathrm{F}{_{\varepsilon
,t}(x},u)\text{ in }\Omega ,\text{ \ }u=0\text{ \ on }\partial \Omega,
\end{equation*}
with 
\begin{equation*}
\mathrm{F}{_{\varepsilon ,t}(x},u)=tf(x,u+\gamma _{\varepsilon
}(u))+(1-t)(\lambda _{1,p}(u^{+})^{p-1}+1),
\end{equation*}
for $t\in \lbrack 0,1]$ and $\varepsilon \in (0,1)$, where $s^{+}:=\max
\{0,s\}.$ It is known that the problem 
\begin{equation*}
(\mathrm{P}_{\varepsilon ,0})\qquad \left\{ 
\begin{array}{l}
-\Delta _{p}u=\lambda _{1,p}(u^{+})^{p-1}+1\text{ in }\Omega \\ 
u=0\text{ \ on }\partial \Omega%
\end{array}
\right.
\end{equation*}
does not admit solutions $u\in W_{0}^{1,p}(\Omega)$ (see \cite[Proposition
9.64]{MMP}).

\begin{proposition}
\label{P2} Assume that $(\mathrm{H.2})$ and $(\mathrm{H.3})$ hold. Then each
solution $u_{\varepsilon }$ of $(\mathrm{P}_{\varepsilon ,t})$ belongs to $%
\mathcal{C}^{1}(\overline{\Omega })$ and satisfies 
\begin{equation}
\left\Vert u_{\varepsilon }\right\Vert _{\mathcal{C}^{1}(\overline{\Omega }%
)}<R_{\varepsilon },  \label{31}
\end{equation}%
for all $t\in (0,1]$ and $\varepsilon \in (0,1)$, provided the constant $%
R_{\varepsilon }>0$ is sufficiently large.
\end{proposition}

\begin{proof}
The nonlinear regularity theory implies that $u_{\varepsilon}\in \mathcal{C}%
^{1}(\overline{\Omega})$. Due to the fact that $(\mathrm{P}_{\varepsilon
,0}) $ has no solution, it suffices to prove that the assertion holds for
every $t\in [\delta,1]$ with any $\delta\in (0,1)$. Suppose by contradiction
that for every $n$ there exist $t_{n}\in [\delta,1]$ and a solution $%
u_{\varepsilon ,n}$ of $(\mathrm{P}_{\varepsilon ,t_{n}})$ such that 
\begin{equation*}
t_{n}\rightarrow t\in [\delta,1]\text{ \ and \ } \theta_{n}:=\Vert
u_{\varepsilon ,n}\Vert_{\mathcal{C}^{1}(\overline{\Omega })}\rightarrow
\infty \text{ \ as }n\rightarrow \infty.
\end{equation*}

Setting 
\begin{equation*}
v_{\varepsilon ,n}:=\frac{u_{\varepsilon ,n}}{\theta_{n}}\in \mathcal{C}^{1}(%
\overline{\Omega }),
\end{equation*}
problem $(\mathrm{P}_{\varepsilon ,t_{n}})$ results in 
\begin{equation}  \label{13}
-\Delta _{p}v_{\varepsilon ,n}=\frac{t_{n}}{\theta _{n}^{p-1}}%
f(x,u_{\varepsilon ,n}+\gamma _{\varepsilon }(u_{\varepsilon ,n}))+\frac{%
1-t_{n}}{\theta _{n}^{p-1}}(\lambda _{1,p}(u_{\varepsilon ,n}^{+})^{p-1}+1).
\end{equation}
By $(\mathrm{H.2})$ and \eqref{12}, we have the estimate 
\begin{eqnarray*}
&&\left\vert \frac{t_{n}}{\theta _{n}^{p-1}}f(x,u_{\varepsilon ,n}+\gamma
_{\varepsilon }(u_{\varepsilon ,n}))+\frac{1-t_{n}}{\theta _{n}^{p-1}}%
(\lambda _{1,p}(u_{\varepsilon ,n}^{+})^{p-1}+1)\right\vert \\
&&\leq M(1+|u_{\varepsilon ,n}+\gamma _{\varepsilon }(u_{\varepsilon
,n})|^{\alpha })+\lambda _{1,p}({\normalsize v}_{\varepsilon ,n}^{+})^{p-1}+1
\\
&&\leq M(1+\varepsilon ^{\alpha })+\lambda _{1,p}({\normalsize v}%
_{\varepsilon ,n}^{+})^{p-1}+1 \\
&&\leq C(1+\Vert {\normalsize v}_{\varepsilon ,n}\Vert _{C^{1}(\overline{%
\Omega })}^{p-1})
\end{eqnarray*}
in $\Omega$, with some constant $C>0$ independent of $n$. Thanks to the
regularity up to the boundary (refer to \cite{L}), we derive from \eqref{13}
that ${\normalsize v}_{\varepsilon ,n}$ is bounded in $\mathcal{C}^{1,\tau }(%
\overline{\Omega })$ with certain $\tau \in (0,1)$. The compactness of the
embedding $\mathcal{C}^{1,\tau }(\overline{\Omega })\subset \mathcal{C}^{1}(%
\overline{\Omega })$ implies along a subsequence that 
\begin{equation}  \label{17}
{\normalsize v}_{\varepsilon ,n}\rightarrow {\normalsize v}_{\varepsilon}%
\text{ \ in }\mathcal{C}^{1}(\overline{\Omega })\text{.}
\end{equation}
Letting $n\rightarrow \infty$ in \eqref{13}, on the basis of \eqref{17} we
arrive at 
\begin{equation*}
\left\{ 
\begin{array}{ll}
-\Delta _{p}{\normalsize v}_{\varepsilon }=(1-t)\lambda _{1,p}({\normalsize v%
}_{\varepsilon }^{+})^{p-1} & \text{ in }\Omega , \\ 
{\normalsize v}_{\varepsilon }=0 & \text{on }\partial \Omega.%
\end{array}
\right.
\end{equation*}
This ensures that ${\normalsize v}_{\varepsilon }={\normalsize v}%
_{\varepsilon }^{+}$, which is nonzero because $\Vert {\normalsize v}%
_{\varepsilon }\Vert _{\mathcal{C}^{1}(\overline{\Omega })}=1$. Acting with $%
{\normalsize v}_{\varepsilon }$ produces 
\begin{equation*}
\int_{\Omega }|\nabla {\normalsize v}_{\varepsilon }|^{p}\,\mathrm{d}%
x=(1-t)\lambda _{1,p}\int_{\Omega }{\normalsize v}_{\varepsilon }^{p}\,%
\mathrm{d}x.
\end{equation*}
We cannot have $t=1$ because ${\normalsize v}_{\varepsilon }\neq 0$. If $t<1$%
, by (\ref{71}) we reach the contradiction 
\begin{equation*}
\lambda _{1,p}\int_{\Omega }{\normalsize v}_{\varepsilon }^{p}\,\mathrm{d}%
x\leq (1-t)\lambda _{1,p}\int_{\Omega }{\normalsize v}_{\varepsilon }^{p}\,%
\mathrm{d}x,
\end{equation*}
which completes the proof.
\end{proof}

For every $\varepsilon \in (0,1)$, let us define the homotopy $\mathrm{H}%
_{\varepsilon }:[0,1]\times \overline{\mathcal{B}}_{R_{\varepsilon
}}(0)\rightarrow \mathcal{C}^{1}(\overline{\Omega })$ by 
\begin{equation*}
\mathrm{H}_{\varepsilon }(t,u)=u-(-\Delta _{p})^{-1}\mathrm{F}{\
_{\varepsilon ,t}(x},u)
\end{equation*}%
that is admissible for the Leray-Schauder topological degree by Proposition %
\ref{P2} and because the operator $(-\Delta _{p})^{-1}$ is compact.

The nonexistence of solutions to problem $(\mathrm{P}_{\varepsilon ,0})$
yields 
\begin{equation*}
\deg \left( \mathrm{H}_{\varepsilon }(0,\cdot ),\mathcal{B}_{R_{\varepsilon
}}(0),0\right) =0.
\end{equation*}
Consequently, the homotopy invariance property implies 
\begin{equation}  \label{35}
\begin{array}{c}
\deg \left( \mathrm{H}_{\varepsilon }(1,\cdot),\mathcal{B}_{R_{\varepsilon
}}(0),0\right)=0\text{ for all }\varepsilon \in (0,1).%
\end{array}%
\end{equation}

\subsubsection{\textbf{The degree on }$\mathcal{B}_{R_{\protect\varepsilon%
}}(0) \backslash \overline{\mathcal{A}}_{R_{\protect\varepsilon }}$\textbf{.}%
}

Consider the problem 
\begin{equation*}
(\mathrm{\tilde{P}}_{\varepsilon ,t})\qquad -\Delta _{p}u=\mathrm{\tilde{F}}{%
_{\varepsilon ,t}}({x,}u)\text{ in }\Omega ,\text{ \ }u=0\text{ \ on }
\partial \Omega ,
\end{equation*}
for all $t\in \lbrack 0,1]$ and $\varepsilon \in (0,1)$, where 
\begin{equation*}
\mathrm{\tilde{F}}{_{\varepsilon ,t}}({x,}u)=tf(x,u+\gamma _{\varepsilon
}(u))+(1-t)\text{ }sgn(u)\text{ }\frac{\phi _{1,p}}{\left\Vert \phi
_{1,p}\right\Vert _{\infty }}.
\end{equation*}

\begin{proposition}
\label{P3} Assume that $(\mathrm{H.2})$ and $(\mathrm{H.3})$ are fulfilled.
Then every solution $u_{\varepsilon }$ of $(\mathrm{\tilde{P}}_{\varepsilon
,t})$ belongs to $\mathcal{C}^{1}(\overline{\Omega })$ and satisfies %
\eqref{31} for all $t\in \lbrack 0,1]$ and $\varepsilon \in (0,1)$, provided
the constant $R_{\varepsilon }>0$ is sufficiently large. In addition, every
positive (resp. negative) solution $u_{\varepsilon ,+}$ (resp. $%
u_{\varepsilon ,-}$) of $(\mathrm{\tilde{P}}_{\varepsilon ,t})$ verifies 
\begin{equation}
\begin{array}{c}
\underline{u}\ll u_{\varepsilon ,+}\text{ \ (resp. }-\underline{u}\gg
u_{\varepsilon ,-}\text{)}\quad \text{in}\ \Omega.%
\end{array}
\label{11}
\end{equation}
\end{proposition}

\begin{proof}
By assumption $(\mathrm{H.2})$, one has 
\begin{equation*}
\begin{array}{l}
|\mathrm{\tilde{F}}{_{\varepsilon ,t}}({x,}u)|\leq M(1+\varepsilon ^{\alpha
})+1.%
\end{array}%
\end{equation*}
The Moser iteration can be implemented ensuring that $u_{\varepsilon }\in
L^{\infty }(\Omega)$. Next, the regularity theory up to the boundary (see 
\cite{L}) entails the bound in \eqref{31}.

We are going to show the first inequality in (\ref{11}). The second one
follows similarly. Let $u_{\varepsilon ,+}$ be a positive solution of $(%
\mathrm{\tilde{P}}_{\varepsilon ,t})$. Given $\delta >0$, on account of $(%
\mathrm{H.3})$, (\ref{14}), (\ref{12}) and \eqref{31}, we have 
\begin{eqnarray*}
&&C^{-(p-1)}d(x)^{\alpha }\leq C^{-(p-1)}\delta ^{\alpha } \\
&<&tm(R_{\varepsilon }+1)^{\beta }+(1-t)\frac{c_{0}\delta }{\left\Vert \phi
_{1,p}\right\Vert _{\infty }} \\
&\leq &tm(u_{\varepsilon ,+}+\gamma _{\varepsilon }(u_{\varepsilon
,+}))^{\beta }+(1-t)\frac{\phi _{1,p}}{\left\Vert \phi _{1,p}\right\Vert
_{\infty }} \\
&\leq &\mathrm{\tilde{F}}{_{\varepsilon ,t}}({x,}u_{\varepsilon ,+})\text{ \
in\ }\Omega \backslash \overline{\Omega }_{\delta }
\end{eqnarray*}
and 
\begin{equation*}
-C^{-(p-1)}<0\leq \mathrm{\tilde{F}}{_{\varepsilon ,t}}({x,}u_{\varepsilon
,+})\text{ in }\Omega _{\delta },
\end{equation*}
for all $t\in \lbrack 0,1]$ and $\varepsilon \in (0,1)$, provided the
constant $C>0$ is large enough. Thus, for each compact set $\mathrm{K}%
\subset \subset \Omega $, there is a constant $\sigma =\tau (\mathrm{K})>0$
such that 
\begin{equation*}
\sigma +\left\{ 
\begin{array}{ll}
d(x)^{\alpha } & \text{in \ }\Omega \setminus \overline{\Omega }_{\delta },
\\ 
-1 & \text{in \ }\Omega _{\delta },%
\end{array}
\right. <\mathrm{\tilde{F}}{_{\varepsilon ,t}}({x,}u_{\varepsilon ,+})\text{
in }\ \mathrm{K}.
\end{equation*}
Hence, via (\ref{20*}), (\ref{24}) and the strong comparison principle (see 
\cite[Proposition $2.6$]{AR}), we infer that $\underline{u}\ll
u_{\varepsilon ,+}$ in $\overline{\Omega }$. Then Theorem \ref{T2} ends the
proof.
\end{proof}

Let us define the homotopy $\mathrm{\tilde{H}}_{\varepsilon }:[0,1]\times 
\overline{\mathcal{B}_{R_{\varepsilon }}(0)\backslash \mathcal{A}%
_{R_{\varepsilon }}}\rightarrow \mathcal{C}^{1}(\overline{\Omega })$ by 
\begin{equation*}
\mathrm{\tilde{H}}_{\varepsilon }(t,u)=u-(-\Delta _{p})^{-1}\mathrm{\tilde{F}%
}{_{\varepsilon ,t}(x},u)
\end{equation*}
for all $t\in \lbrack 0,1]$ and $\varepsilon \in (0,1)$. The homotopy $%
\mathrm{\tilde{H}}_{\varepsilon }$ is admissible for the Leray-Schauder
topological degree by Proposition \ref{P3} and because the operator $%
(-\Delta _{p})^{-1}$ is compact.

The Minty-Browder Theorem (see, e.g., \cite[Theorem 5.16]{B}) and the
regularity result in \cite{L} guarantee that the problem 
\begin{equation*}
-\Delta _{p}u=\frac{\phi _{1,p}}{\left\Vert \phi _{1,p}\right\Vert _{\infty }%
}\text{ in }\Omega,\text{ }u=0\text{ on }\partial \Omega ,
\end{equation*}
admits a unique solution $u\in \mathcal{C}^{1}(\overline{\Omega })$. Then
the homotopy invariance property of the Leray-Schauder degree shows 
\begin{equation*}
\begin{array}{ll}
\deg (\mathrm{\tilde{H}}_{\varepsilon }(1,\cdot ),\mathcal{B}%
_{R_{\varepsilon }}(0)\backslash \overline{\mathcal{A}}_{R_{\varepsilon }},0)
& =\deg (\mathrm{\tilde{H}}_{\varepsilon }(0,\cdot ),\mathcal{B}%
_{R_{\varepsilon }}(0)\backslash \overline{\mathcal{A}}_{R_{\varepsilon }},0)
\\ 
& =1.%
\end{array}%
\end{equation*}
Since 
\begin{equation*}
\mathrm{H}_{\varepsilon }(1,\cdot )=\mathrm{\tilde{H}}_{\varepsilon
}(1,\cdot )\quad \text{in}\,\,\,\mathcal{B}_{R_{\varepsilon }}(0)\backslash 
\mathcal{A}_{R_{\varepsilon }},
\end{equation*}
we deduce that 
\begin{equation}
\begin{array}{c}
\deg (\mathrm{H}_{\varepsilon }(1,\cdot),\mathcal{B}_{R_{\varepsilon
}}(0)\backslash \overline{\mathcal{A}}_{R_{\varepsilon }},0)=1.%
\end{array}
\label{56}
\end{equation}

\subsubsection{\textbf{Proof of Theorem \protect\ref{T4}.}}

We may assume that 
\begin{equation*}
\mathrm{H}_{\varepsilon }(1,u)\not=0\quad \text{for all }u\in \partial 
\mathcal{A}_{R_{\varepsilon }}\text{ and }\varepsilon \in (0,1).
\end{equation*}%
Otherwise, $u\in \partial \mathcal{A}_{R_{\varepsilon }}$ would be a
solution of $(\mathrm{P}_{\varepsilon })$ within $\left[ -\underline{u},%
\underline{u}\right] $ and thus Theorem \ref{T4} is proven.

By virtue of the domain additivity property of Leray-Schauder degree it
follows that 
\begin{eqnarray*}
&&\deg (\mathrm{H}_{\varepsilon }(1,\cdot ),\mathcal{B}_{R_{\varepsilon
}}(0)\backslash \overline{\mathcal{A}}_{R_{\varepsilon }},0)+\deg (\mathrm{H}%
_{\varepsilon }(1,\cdot), \mathcal{A}_{R_{\varepsilon }},0) \\
&=&\deg (\mathrm{H}_{\varepsilon }(1,\cdot ),\mathcal{B}_{R_{\varepsilon
}}(0),0).
\end{eqnarray*}
Hence, by (\ref{35}) and (\ref{56}), we deduce 
\begin{equation*}
\deg (\mathrm{H}_{\varepsilon }(1,\cdot ),\mathcal{A}_{R_{\varepsilon
}},0)=-1,
\end{equation*}
obtaining that problem $(\mathrm{P}_{\varepsilon })$ has a solution $%
u_{\varepsilon }\in \mathcal{A}_{R_{\varepsilon }},$ for all $\varepsilon
\in (0,1)$. Furthermore, the nonlinear regularity theory up to the boundary
(see \cite{L}) guarantees that $u_{\varepsilon }\in \mathcal{C}^{1,\tau}(%
\overline{\Omega })$ with some $\tau \in (0,1)$.

\subsection{\textbf{Proof of Theorem \protect\ref{T3}.}}

Set $\varepsilon =\frac{1}{n}$ in $(\mathrm{P}_{\varepsilon })$ with any
integer $n\geq 1$. Theorem \ref{T4} provides a solution $u_{n}:=u_{\frac{1}{n%
}}$ of problem $(\mathrm{P}_{\frac{1}{n}})$ such that 
\begin{equation}
u_{n}\in \lbrack -\underline{u},\underline{u}],\quad \forall n.  \label{26}
\end{equation}
There it holds 
\begin{equation}
\int_{\Omega }|\nabla u_{n}|^{p-2}\nabla u_{n}\text{\thinspace }\nabla
\varphi \text{ }\mathrm{d}x=\int_{\Omega }f(x,u_{n}+\gamma _{\frac{1}{n}%
})\varphi \ \mathrm{d}x  \label{1}
\end{equation}
for all $\varphi \in W_{0}^{1,p}(\Omega )$. Acting with $\varphi =u_{n}$, we
find from $(\mathrm{H.2})$ and (\ref{26}) the estimate 
\begin{eqnarray*}
&&\int_{\Omega }|\nabla u_{n}|^{p}\text{ }\mathrm{d}x=\int_{\Omega
}f(x,u_{n}+\gamma _{\frac{1}{n}}(u_{n}))u_{n}\ \mathrm{d}x \\
&\leq &M\int_{\Omega }(1+|u_{n}+\gamma _{\frac{1}{n}}(u_{n})|^{\alpha
})|u_{n}|\ \mathrm{d}x \\
&\leq &M\int_{\Omega }(|u_{n}|+|u_{n}|^{\alpha +1})\ \mathrm{d}x \\
&\leq &M|\Omega |(\left\Vert \overline{u}\right\Vert _{\infty }+\left\Vert 
\overline{u}\right\Vert _{\infty }^{\alpha +1})<\infty .
\end{eqnarray*}
Hence, along a subsequence it holds $u_{n}\rightharpoonup u_{\ast }$ in $%
W_{0}^{1,p}(\Omega )$ with some $u_{\ast }\in W_{0}^{1,p}(\Omega )$. By %
\eqref{26}, we note that 
\begin{equation}
-\underline{u}\leq u_{\ast }\leq \underline{u}\ \text{ in }\Omega.
\label{112}
\end{equation}
Inserting $\varphi =u_{n}-u_{\ast }$ yields 
\begin{equation*}
\int_{\Omega }|\nabla u_{n}|^{p-2}\nabla u_{n}\text{\thinspace }\nabla
(u_{n}-u_{\ast })\text{ }\mathrm{d}x=\int_{\Omega }f(x,u_{n}+\gamma _{\frac{1%
}{n}}(u_{n}))(u_{n}-u_{\ast })\ \mathrm{d}x.
\end{equation*}%
Then, from $(\mathrm{H.2})$ and \eqref{26}, we see that 
\begin{equation*}
\underset{n\rightarrow \infty }{\lim }\int_{\Omega }|\nabla
u_{n}|^{p-2}\nabla u_{n}\text{\thinspace }\nabla (u_{n}-u_{\ast })\text{ }%
\mathrm{d}x=0.
\end{equation*}
At this point, the $S_{+}$-property of $-\Delta _{p}$ on $W_{0}^{1,p}(\Omega
)$ (see, e.g., \cite[Proposition 3.5]{MMP}) implies 
\begin{equation*}
u_{n}\rightarrow u_{\ast }\text{ in }W_{0}^{1,p}(\Omega ).
\end{equation*}
By virtue of $(\mathrm{H.2})$, Lebesgue's dominated convergence theorem
entails 
\begin{equation*}
\lim_{n\rightarrow \infty }\int_{\Omega }f(x,u_{n}+\gamma _{\frac{1}{n}%
}(u_{n}))\varphi \ \mathrm{d}x=\int_{\Omega }f(x,u_{\ast })\varphi \ \mathrm{%
d}x,
\end{equation*}%
for all $\varphi \in W_{0}^{1,p}(\Omega )$. Consequently, we may pass to the
limit in (\ref{1}) finding that $u_{\ast }$ is a solution of problem $(%
\mathrm{P})$. Remark \ref{R1} renders that the solution $u_{\ast }$ is
nontrivial. Moreover, due to \eqref{11}, \eqref{112} and Theorems \ref{T1}
and \ref{T2}, we infer that the solution $u_{\ast }$ of problem $(\mathrm{P}%
) $ cannot be of constant sign (otherwise the opposite constant-sign
solutions $u_{-}^{\ast}$ and $u_{+}^{\ast }$ were not extremal as known from
Theorem \ref{T2}). Therefore the solution $u_{\ast}$ is nodal. This
completes the proof. \newline

\end{document}